\documentclass[a4paper]{article}
\usepackage{amsthm,amssymb,amsmath,enumerate,graphicx,epsf,bbold}

\newcommand{\COLORON}{1}
\newcommand{\NOTESON}{0}
\newcommand{\Debug}{0}
\usepackage[usenames,dvips]{color} 
\usepackage{amsthm,amssymb,amsmath,enumerate,graphicx,epsf,mathbbol,stmaryrd}
\usepackage[dvips, bookmarks, colorlinks=false]{hyperref}

\newcommand{\comment}[1]{}
\newcommand{\COMMENT}[1]{}

\definecolor{darkgray}{rgb}{0.3,0.3,0.3}
\newcommand{\defi}[1]{{\color{darkgray}\emph{#1}}}

\newcommand{\acknowledgement}{\section*{Acknowledgement}}

\comment{
	\begin{lemma}\label{}	
\end{lemma}
\begin{proof}

\end{proof}

\begin{theorem}\label{}
\end{theorem} 
\begin{proof} 	

\end{proof}

}

\newtheorem{proposition}{Proposition}[section]

\newtheorem{theorem}[proposition]{Theorem}
\newtheorem{corollary}[proposition]{Corollary}
\newtheorem{lemma}[proposition]{Lemma}
\newtheorem{observation}[proposition]{Observation}
\newtheorem{conjecture}{{Conjecture}}[section]

\newtheorem{problem}[conjecture]{{Problem}}

\newtheorem{examp}[proposition]{Example}

\newcommand{\FIG}{0}

\ifnum \NOTESON = 1 \newcommand{\note}[1]{ 

	\ 

	{\color{blue} \hspace*{-60pt} NOTE: \color{Turquoise}{\small  \tt \begin{minipage}[c]{1.1\textwidth}  #1 \end{minipage} \ignorespacesafterend }} 
	
	\ 
	
	}
\else \newcommand{\note}[1]{} \fi

\ifnum \Debug = 1 
\else  \fi

\ifnum \FIG = 1 \newcommand{\fig}[1]{Figure ``{#1}''}
\else \newcommand{\fig}[1]{Figure~\ref{#1}} \fi

\ifnum \Debug = 1 \usepackage[notref,notcite]{showkeys}
\fi

\ifnum \COLORON = 0 \renewcommand{\color}[1]{}
\fi

\newcommand{\showFig}[2]{
   \begin{figure}[htbp]
   \centering
   \noindent
   \epsfbox{#1.eps}
   \caption{\small #2}
   \label{#1}
   \end{figure}
}

\newcommand{\N}{\ensuremath{\mathbb N}}

\newcommand{\sm}{\backslash}

\newcommand{\sgl}[1]{\ensuremath{\{#1\}}}
\newcommand{\pth}[2]{\ensuremath{#1}\text{--}\ensuremath{#2}~path}

\newcommand{\Lr}[1]{Lemma~\ref{#1}}
\newcommand{\Tr}[1]{Theorem~\ref{#1}}

\newcommand{\Prr}[1]{Pro\-position~\ref{#1}}

\newcommand{\Cr}[1]{Corollary~\ref{#1}}
\newcommand{\Cnr}[1]{Con\-jecture~\ref{#1}}

\newcommand{\hcy}{Hamilton circle}

\newcommand{\Fe}{For every}

\newcommand{\st}{such that}
\newcommand{\sot}{so that}
\newcommand{\ti}{there is}
\newcommand{\tho}{there holds}

\newcommand{\obda}{without loss of generality}

\newcommand{\labtequ}[2]{ \begin{equation} \label{#1} 	\begin{minipage}[c]{0.9\textwidth}  #2 \end{minipage} \ignorespacesafterend \end{equation} }

\newcommand{\mySection}[2]{}

\title{On 3-coloured tournaments}
\author{Agelos Georgakopoulos\thanks{Supported by a GIF grant.}\\[3pt]
Philipp Spr\"ussel
\bigskip \\
  {Mathematisches Seminar}\\
  {Universit\"at Hamburg}\\
  {\small Bundesstr.\ 55, 20146, Germany}
}

\renewcommand{\hcy}{Hamilton cycle}

\renewcommand{\rm}{\ensuremath{R^-(x)}}
\newcommand{\rp}{\ensuremath{R^+(x)}}
\newcommand{\bm}{\ensuremath{B^-(x)}}
\newcommand{\bp}{\ensuremath{B^+(x)}}

\newcommand{\rrm}{\ensuremath{R_r^-}}
\newcommand{\rbm}{\ensuremath{R_b^-}}
\newcommand{\rrp}{\ensuremath{R_r^+}}
\newcommand{\rbp}{\ensuremath{R_b^+}}
\newcommand{\brm}{\ensuremath{B_r^-}}
\newcommand{\bbm}{\ensuremath{B_b^-}}
\newcommand{\brp}{\ensuremath{B_r^+}}
\newcommand{\bbp}{\ensuremath{B_b^+}}

\newcommand{\dor}{\ensuremath{\color{red}\longmapsto_r \color{black}}}
\newcommand{\nodor}{\ensuremath{\color{red}\not\longmapsto_r \color{black}}}
\newcommand{\dob}{\ensuremath{\color{blue}\longmapsto_b\color{black}}}
\newcommand{\nodob}{\ensuremath{\color{blue}\not\longmapsto_b\color{black}}}
\newcommand{\doy}{\ensuremath{\color{\yel}\longmapsto_g\color{black}}}
\newcommand{\nodoy}{\ensuremath{\color{\yel}\not\longmapsto_g\color{black}}}
\newcommand{\dom}{\ensuremath{\longmapsto}}
\newcommand{\nodom}{\ensuremath{\not\longmapsto}}

\newcommand{\odor}{\ensuremath{\color{red}\hookrightarrow_r\color{black}}}
\newcommand{\odob}{\ensuremath{\color{blue}\hookrightarrow_b\color{black}}}
\newcommand{\odoy}{\ensuremath{\color{\yel}\hookrightarrow_g\color{black}}}

\newcommand{\beats}{\to}
\newcommand{\btr}{\color{red} \beats_r\color{black}}
\newcommand{\btb}{\color{blue} \beats_b\color{black}}
\newcommand{\bty}{\color{\yel} \beats_g\color{black}}

\newcommand{\noproof}{\unskip\nobreak\hfill\penalty50\hskip2em\hbox{}\nobreak\hfill%
       $\square$\parfillskip=0pt\finalhyphendemerits=0\par}

\newcommand{\yel}{green}

\begin{document}
\maketitle

\begin{abstract}
We (re-)prove that in every 3-edge-coloured tournament in which no vertex is incident with all colours there is either a cyclic rainbow triangle or a vertex dominating every other vertex monochromatically.
\end{abstract}

\section{Introduction}

It is an easy and well-known fact that in every finite tournament there is a vertex that dominates every other vertex, where we say that $x$ \defi{dominates} $y$ if \ti\ a directed path from $x$ to $y$. Sands, Sauer, and Woodrow~\cite{SSW} generalised this fact to \defi{$2$-coloured} tournaments, i.e.\ tournament the edges of which are coloured with (at most) $2$ colours: they proved that in every finite $2$-coloured tournament there is a vertex that dominates every other vertex monochromatically, where we say that $x$ \defi{dominates} $y$ \defi{monochromatically} if there is a directed path from $x$ to $y$ all edges of which have the same colour. (In fact, their theorem is much more general, and follows from a result about infinite 2-coloured  directed graphs.)

If we allow three or more colours then the situation becomes much more complicated, and the above assertion does not remain true: in the non-transitive tournament on three vertices whose edges have three distinct colours no vertex dominates both other vertices monochromatically. We call such a tournament a \defi{$T_3$}. Motivated by this and other examples, Sands, Sauer, and Woodrow~\cite{SSW} posed the following problem, which they also attribute to Erd\H{o}s.

\begin{problem}[\cite{SSW}] 
\Fe\ $n$, is there a (least) integer $f(n)$ \sot\ every finite $n$-coloured tournament $T$ has a set $S$ of $f(n)$ vertices \st\ for every vertex $y$ of $T$ \ti\ a vertex in $S$ that dominates $y$ monochromatically? In particular, is $f(3)=3$?
\end{problem}

A further related problem they pose is

\begin{conjecture}[\cite{SSW}] \label{conj}
Let $T$ be a finite $3$-coloured tournament. Then $T$ has either a triple of vertices that span a $T_3$ or a single vertex that dominates every other vertex monochromatically.
\end{conjecture}



Shen Minggang \cite{Minggang} proved a weaker version of \Cnr{conj}, stating that every 3-coloured tournament contains either a rainbow triangle or it has a vertex that dominates every other vertex monochromatically. For a survey about this problem, and tournaments in general, see \cite{JeGuSur}.

In \cite{SaMoMon} \Cnr{conj} was proved for the special case in which each vertex meets at most two of the three colours:

\begin{theorem}[\cite{SaMoMon}]\label{main}
Let $T$ be a 3-coloured tournament in which each vertex is incident with edges of at most two colours. Then $T$ has either a triple of vertices that span a $T_3$ or a single vertex that dominates every other vertex monochromatically.
\end{theorem} 

The proof of \Tr{main} in \cite{SaMoMon} contained a long case distinction. It is the main aim of this paper to give an alternative, perhaps more elegant proof of \Tr{main}.
Our proof is elementary, and makes use of an elegant observation of \cite{Minggang} stating that if Conjecture~\ref{conj} is false then any minimal counterexample has a directed Hamilton cycle $C$ \st\ each vertex  monochromatically dominates every other vertex except for its predecessor on $H$, see \Lr{genhamilton}.

\section{The \hcy}

Every tournament in this paper will be finite and 3-coloured; the colours will always be red, blue, and \yel. If a vertex $x$ dominates a vertex $y$ monochromatically we write $x\dom y$. From now on we will sometimes just write \defi{dominates} instead of ``dominates monochromatically''. If the edge between $x$ and $y$ is directed from $x$ to $y$ we say that $x$ \defi{beats} $y$ and write $x\beats y$. We endow the symbols `$\dom$' and `$\beats$' with an index $\color{red}_r$, $\color{blue}_b$, or $\color{\yel}_g$ to assert that the domination or edge is in red, blue, or \yel\ colour respectively. We further write $x\odor y$ if $x$ dominates $y$ \defi{only} in red and $x\nodor$ if $x$ does {not} dominate $y$ in red, and similarly for blue and \yel. If a vertex $x$ dominates all vertices in a tournament, we abbreviate this fact by saying that $x$ \emph{dominates} the tournament.

If $D$ is a tournament and $U$ a subset of its vertices, then we denote by $D[U]$ the subtournament of $D$ spanned by the vertices in $U$.

If $C$ is a directed path or cycle and $x,y$ are two of its vertices then $xCy$ denotes the subpath of $C$ from $x$ to $y$.

For completeness we reprove the following result of Shen Minggang mentioned in the introduction. 

\begin{lemma}[\cite{Minggang}]\label{genhamilton}
  If $D$ is a minimal counterexample (with respect to containment) to Conjecture~\ref{conj} then it has a (unique) directed Hamilton cycle $C$ \st\ each vertex  monochromatically dominates every vertex except for its predecessor on $C$.
\end{lemma}

\begin{proof}
  Since no vertex in $D$ dominates every other vertex, it is not hard to find a directed cycle $C$ in $D$ \st\ no vertex in $V(C)$  dominates its predecessor on $C$.
  \COMMENT{Start with a vertex $v_0$, then recursively let $v_{i+1}$ be a vertex not dominated by $v_i$. Stop the construction when $v_i=v_j$ for $i>j$. (One easily sees that $i-j>2$ in that case.) Then $v_i,v_{i-1},\dotsc,v_{j+1}$ form the desired cycle.}%
  It is easy to see that the subtournament $D[V(C)]$ is also a counterexample to Conjecture~\ref{conj}. This, and the minimality of $D$, implies that $C$ is a Hamilton cycle.
  
  Now suppose that the vertices $v,w\in V(C)$ are not consecutive on $C$ and that $v\nodom w$. Then $w \beats v$ must hold, and the union of $vCw$ with the edge  $wv$ is a directed cycle $C'$ shorter than $C$ on which no vertex  dominates its predecessor. By our previous argument, $C'$ contradicts the choice of $D$ as a minimal counterexample since $D[V(C')]$ is also a counterexample. This means that each vertex of $D$  dominates all vertices but its predecessor on the Hamilton cycle $C$. Easily, no other Hamilton cycle of $D$ (up to rotation) can have the latter property.
\end{proof}

It is straightforward to check that every minimal (with respect to inclusion) counterexample to the statement of Theorem~\ref{main} is also a minimal counterexample to Conjecture~\ref{conj}. (Note though, that a {\em minimum} counterexample to the statement of Theorem~\ref{main} need not be a {\em minimum} counterexample to Conjecture~\ref{conj}.) Thus, \Lr{genhamilton} implies

\begin{corollary}\label{hamilton}
  If $D$ is a minimal counterexample to the statement of Theorem~\ref{main} then it has a (unique) directed Hamilton cycle $C$ \st\ each vertex  monochromatically dominates every vertex except for its predecessor on $C$.
\end{corollary}

For the rest of this section let $D$ be a minimal counterexample to Conjecture~\ref{conj} and let $C$ be the Hamilton cycle provided by Lemma~\ref{genhamilton}. In this paper, we will use the results of this section only for the case that $D$ is even a counterexample to the statement of Theorem~\ref{main}, but we state them in greater generality in order to keep them accessible for the general case.

By assumption, no vertex in $D$ dominates every other vertex monochromatically. Conversely, as no vertex in $D$ is dominated by its successor on $C$, no vertex in $D$ is  dominated by every other vertex. It turns out that $D$ is also minimal with that property:

\begin{lemma}\label{reverse}
  $D$ is minimal with the property that it contains no vertex that is dominated by every other vertex.
\end{lemma}

\begin{proof}
  All that remains to check is that every proper non-empty subtournament $D_0$ of $D$ contains a vertex that is dominated by every other vertex in $D_0$. By the minimality of $D$ as a counterexample to Conjecture~\ref{conj}, \ti\ a vertex $x_1$ that dominates $D_0$. Similarly, \ti\ a vertex $x_2$ that dominates $D_1:= D_0 - x_1$. Continuing like this, we find a sequence $x_1,x_2,\dotsc,x_n$ (with $n=|V(D_0)|$) such that each $x_i$ dominates $D_{i-1}$, where $D_i=D_{i-1}-x_i$. Then $x_n$ is dominated by every vertex in $D_0$.
\end{proof}

Given a vertex $x$, we write $x^+ = x^{+1}$ for its successor and $x^- = x^{-1}$ for its predecessor on $C$. Then, recursively for $i=1,2,\dotsc$, let $x^{+(i+1)}$ be the successor of $x^{+i}$ on $C$ and let $x^{-(i+1)}$ be the predecessor of $x^{-i}$ on $C$.

As a first step towards \Tr{main} we prove 
\begin{proposition}\label{1in}
  $D$ has  no vertex all incoming (or all outgoing) edges of which have the same colour.
\end{proposition}

\begin{proof}
  Suppose that there is a vertex $x$ with only red edges coming in. Let $D'$ be obtained from $D$ by deleting $x$ and all vertices that $x$ sends a red edge to. Note that $D'$ is not empty as it contains $x^-$. By the minimality of $D$ there is a vertex $y\in D'$ dominating $D'$. If $y$ beats $x$ then it does so in red and thus it dominates also the vertices in $V(D)\setminus V(D')$, a contradiction to the fact that no vertex dominates $D$. For the same reason, $y\nodor x$, hence  $y\nodor x^-$.
  
  Thus $x$ beats $y$, and by the definition of $D'$ it does so in blue or \yel. Without loss of generality, $x\btb y$. If $y\dob x^-$ then $x\dob x^-$, a contradiction to the choice of $C$. Thus $y\odoy x^-$. Let $P$ be a \yel\ \pth{y}{x^-}\ in $D'$. If $P$ contains a vertex $p$ \st\ $x \bty p$ then $x\doy x^-$, again a contradiction. Since $x$ has only red edges coming in, and it sends no red edges to vertices in $D'$, this means that every vertex on $P\leq D'$ either beats $x$ in red or is beaten by $x$ in blue. As the first vertex $y$ of $P$ is beaten by $x$ in blue, while the last vertex $x^-$ of $P$ beats $x$ in red, $P$ contains an edge from a vertex $v$ that is beaten by $x$ in blue to a vertex $w$ that beats $x$ in red. But then $x,v,w$ span a $T_3$, a contradiction.
  
  Hence there is no vertex all of whose incoming edges have the same colour. Inverting all edges and repeating the argument shows that there is also no vertex all of whose outgoing edges have the same colour.
\end{proof}

By the choice of $C$, a vertex $x$ dominates every other vertex $y\not= x^-$. The following lemma tells us that $C$ not only supplies information about the existence or not of a domination, but also encodes a lot of information about how each domination is implemented.

\begin{lemma} \label{xCy}
  For every $x,y \in V(G)$ with $y\not= x^-$, $x$ dominates $y$ in $D[xCy]$.
\end{lemma}

\begin{proof}
  By the minimality of $D$ there is a vertex in $D[xCy]$ dominating all other vertices. Since for every vertex $z\not= x$ in $xCy$ its predecessor $z^-$ is also contained in $xCy$, this vertex can only be $x$.
\end{proof}

Note that this does {not} mean that if $x\dor y$ in $D$, then also $x\dor y$ in $D[xCy]$: if $x$ also dominates $y$ in some other colour except red, then it could be the case that $x$ dominates $y$ in $D[xCy]$ only in that colour. However, if $x\odor y$ in $D$ then also $x\odor y$ in $D[xCy]$.

By Lemma~\ref{genhamilton}, for every vertex $y\in V(D)\setminus\{x,x^+\}$ both dominations $x^+\dom y$ and $y \dom x$ take place. These dominations cannot be in the same colour, as  $x^+$ would then dominate $x$ in that colour. We have proved

\begin{observation}\label{xxplus}
  No vertex can dominate $x$ in a colour in which it is dominated by $x^+$. \noproof
\end{observation}

Trivially, $C$ cannot be monochromatic. Therefore it contains consecutive edges with distinct colours. The following lemma tells us how the edges and dominations in $D$ behave at such points.

\begin{lemma}\label{twoCedges}
  Suppose that the edges $x^-x$ and $xx^+$ have distinct colours. Then $x^-\beats x^+$ and $x^+$ dominates $x^-$ only in the third colour.
\end{lemma}

\begin{proof}
  Suppose, without loss of generality, that $x^-\btr x$ and $x\btb x^+$. Applying Observation~\ref{xxplus} twice---once for $x$ and its successor $x^+$ and once for $x^-$ and its successor $x$---we obtain $x^+\odoy x^-$. Thus, if $x^+\beats x^-$, then $x^+\bty x^-$ and $x^-,x,x^+$ would form a $T_3$; hence $x^-\beats x^+$.
\end{proof}

At first sight it might seem that the existence of a Hamilton cycle $C$ as in \Lr{genhamilton} with so strong properties would quickly lead to a contradiction, but apparently this is not the case. Even under very strong assumptions about the distribution of colours on $C$ it is very hard to make any progress; as a piece of evidence about this, we prove here that the edges of $C$ cannot alternate between two colours. We could not prove that they cannot alternate between three colours.

\begin{proposition}\label{alternate}
Pick a vertex $z$ of $D$. It is not the case that all edges $z^{+2k}z^{+2k+1}$ with $k\in\N$ are red and  all edges $z^{+2k-1}z^{+2k}$ are blue.
\end{proposition}

\begin{proof}
  Suppose it is. Then clearly $|V(D)|$ is even. Moreover, by Lemma~\ref{twoCedges}, for every vertex $x$ \tho\ $x \odoy x^{-2}$. Hence for every $i\in\N$, $x \doy x^{-2i}$ holds. Thus $V(D)$ decomposes into two sets $V_1,V_2$, each containing every second vertex on $C$, and each vertex in $V_i$ dominates every other vertex in $V_i$ in \yel.
  
  No domination between $V_1$ and $V_2$ can be \yel, for if $v\doy w$ for $v\in V_1$ and $w\in V_2$ (or vice versa), then $v\doy w \doy v^-$ as $v^-\in V_2$. We claim that every vertex $x$ dominates $x^{-3}$ only in the colour of the edge $xx^+$. Indeed, assume \obda\ that $x\btr x^+$; as $x$ and $x^{-3}$ do not lie in the same $V_i$, we have $x\nodoy x^{-3}$. On the other hand, if $x\dob x^{-3}$, then using the fact that the edges of $C$ alternate between blue and red we obtain
  $$x^-\btb x \dob x^{-3} \btb x^{-2},$$
  a contradiction as $x^-$ cannot dominate its predecessor $x^{-2}$.
  
  Thus, still assuming that $x\btr x^+$, we have $x^{+2}\odor x^-$ and $x^{+3}\odob x$. Hence $x\odoy x^{+2}$ because otherwise $x\dor x^-$ or $x^{+3}\dob x^{+2}$. By Lemma~\ref{twoCedges} we thus have $x\bty x^{+2}$. By the same argument we obtain $y\bty y^{+2}$ for every $y\in V(D)$.
  
  Since $x$ does not beat $x^-$, \ti\ a smallest integer $m$  for which $x$ does not beat $x^{+(2m+1)}$; obviously, $m\ge 1$. We claim that $x, x^{+(2m-1)}, x^{+(2m+1)}$ span a $T_3$. We have just shown that $x^{+(2m-1)}\bty x^{+(2m+1)}$, and by the choice of $m$, we have $x\beats x^{+(2m-1)}$ and $x^{+(2m+1)}\beats x$. None of the edges $xx^{+(2m-1)}$ and $x^{+(2m+1)}x$ is \yel\ since $x^{+(2m-1)}$ and $x^{+(2m+1)}$ do not lie in the same $V_i$ as $x$ does. Moreover, these two edges cannot both be red (respectively blue) as otherwise $x^{+(2m+1)} \dor x \dor x^{+(2m-1)}$ would contradict the fact that $x^{+(2m+1)}\odoy x^{+(2m-1)}$ by Lemma~\ref{twoCedges}. This shows that $x, x^{+(2m-1)}, x^{+(2m+1)}$ span a $T_3$ as claimed, which is a contradiction to the choice of $D$.
\end{proof}

\begin{problem}
Pick a vertex $z$ of $D$. Can it be the case that all edges $z^{3k}z^{3k+1}$ with $k\in\N$ are red,  all edges $z^{3k+1}z^{3k+2}$ are \yel, and all edges $z^{3k-1}z^{3k}$ are blue?
\end{problem}

\section{Proof of \Tr{main}}

In this section we prove
\begin{theorem}\label{3colours}
  In a minimal counterexample to Conjecture~\ref{conj} every vertex has incident edges in all three colours.
\end{theorem}
This immediately implies our main result \Tr{main}, see our comment preceding \Cr{hamilton}.

For the rest of the paper let $D$ be a minimal counterexample to Conjecture~\ref{conj} and suppose there is a vertex $x$ for which one colour, say \yel, does not appear among the incident edges. We prove that this cannot be the case. Let $C$ be the Hamilton cycle provided by Lemma~\ref{genhamilton}.

Let \rp\ (resp.\ \rm) be the set of vertices that $x$ sends a red edge to (resp.\ receives a red edge from). Define \bp\ and \bm\ similarly for blue. By Proposition~\ref{1in}, all sets \rp, \rm, \bp, and \bm\ are nonempty.

\comment{
\begin{lemma}
There is no red edge from $B^-(x)$ to $R^-(x)$ and no blue edge from $R^-(x)$ to $B^-(x)$.
\end{lemma}

\begin{proof}
Suppose there is a red edge from $v \in B^-(x)$ to $R^-(x)$. Then $v$ sends both a blue and a red path to $x$, i.e.\ it dominates every vertex dominated by $x$, thus it dominates all vertices.
\end{proof}
}

\ifnum \COLORON = 1 \newcommand{\figc}{figC}
\newcommand{\ourcap}{The sets $R_i^\pm$ and $B_i^\pm$.}
\else \newcommand{\figc}{figBW}
\newcommand{\ourcap}{The sets $R_i^\pm$ and $B_i^\pm$. The dotted lines represent red edges, the dashed lines (resp.\ curves) represent blue edges (resp.\ paths).}
\fi

Define $R^-_i(x)= \{v \in \rm \mid x \dom_i v\}$ and $R^+_i(x)= \{v \in \rp \mid v \dom_i x\}$ for $i\in \{r,b\}$; 
Similarly for the sets  $B^-(x)$ and $B^+(x)$. For example, $\bbm= \{v \in \bm \mid x \dob v\}$; see \fig{\figc}. Note that the assertion $\rrm\not=\emptyset$ is equivalent to the existence of a red directed cycle through $x$, which is in turn equivalent to $\rrp\not=\emptyset$. A similar assertion holds for \bbm, \bbp, and blue cycles through $x$.

\showFig{\figc}{\ourcap}

The following lemma tells us that the sets \rrm, \rbm, \rrp, \rbp, \brm, \bbm, \brp, and \bbp\ are pairwise disjoint.

\begin{lemma}\label{disjoint}
  No vertex dominates $x$ both in red and blue. No vertex is dominated by $x$ both in red and blue.
\end{lemma}

\begin{proof}
  Let $y\in V(D)\sm\{x,x^+\}$. Then $x$ dominates $y^-$ in red or in blue, say $x\dor y^-$. This means that $y\nodor x$ and hence $y\odob x$. Analogously, every vertex $z\in V(D)\sm\{x^-,x\}$ is dominated by $x$ only in red or only in blue.
\end{proof}

Recall that $x^-\beats x$. From now on we assume, without loss of generality, that 
\labtequ{annahm}{$x^- \in \rm$.}


\begin{lemma} \label{m}
\bbp\ and \bbm\ are nonempty.
\end{lemma}

\begin{proof}
Suppose not; then, by our comment before \Lr{disjoint} both sets are empty. As \bbm\ is empty, $x \odor y$ holds for every vertex $y \in \bm$. Now consider the tournament $F:=D-\bm$. As \bm\ is nonempty, the minimality of $D$ implies that there is a vertex $z$ in $F$ that dominates $F$.  Note that $z\not=x$, since $x^-\in V(F)$ by \eqref{annahm}. Thus $z\dor x$, since all incoming edges of $x$ in $F$ are red. But then $z\dor x\dor y$ holds (in $D$) for every vertex $y \in \bm$, which means that $z$ dominates every vertex in $V(F) \cup \bm= V(D)$. This contradicts the fact that no vertex dominates all vertices of $D$. Hence, \bbp\ and \bbm\ are indeed nonempty.
\end{proof}

It might seem at first sight that \Lr{m} implies, by symmetry, that $\rrp,\rrm$ are also non-empty. This argument is however faulty, since we are assuming \eqref{annahm}. Still, using \Lr{m}, we can prove

\begin{lemma} \label{n}
\rrp\ and \rrm\ are nonempty. Moreover, there is a vertex in \rrm\ whose successor on $C$ does not lie in \bp.
\end{lemma}

\begin{proof}
By the minimality of $D$, the subtournament $G:=D-(\rm\cup\bp)$ contains a vertex $z$ dominating it. In this subtournament $x$ can only dominate vertices in red. By Lemma~\ref{m}, $G$ contains a vertex $y$ (in \bm) with $x\dob y$ in $D$. By Lemma~\ref{disjoint} we have $x\odob y$ and hence $x$ does not dominate $y$ in $G$. This means that $z\not=x$, thus $x \dom z^-$. Since all incoming edges at $x$ in $G$ are blue, we obtain $z \dob x$ and hence $x \nodob z^-$ by Observation~\ref{xxplus}. Thus $x\dor z^-$. As $z^-$ is not dominated by $z$, it cannot lie in $\rp \cup \sgl{x} \cup \bm$. Therefore $z^-\in\rm$ and hence $z^-\in \rrm$. Thus \rrm, and hence also \rrp, is nonempty.
\end{proof}

Our last two lemmas prove the existence of vertices $m \in \bbp$ and $n \in \rrm$. We will now make use of this fact to gain some information about $C$.

\begin{proposition} \label{r}
For every pair of vertices $m \in \bbp$ and $n \in \rrm$ with $m\not= n^+$, the path $mCn$ contains a vertex $p \in \rrp$ such that $mCp$ does not meet \bm.
\end{proposition}

\begin{proof}
By Lemma~\ref{disjoint} we have $m\nodor n$ and $m\nodob n$, since otherwise $m\in \brp$ or $n\in \rbm$ respectively. Thus $m \odoy n$, and by \Lr{xCy} we have $m\odoy n$ also in $D[mCn]$. Let $P$ be a \yel\ directed path in $D[mCn]$ from $m$ to $n$. If there was a \yel\ edge from some vertex $y\in\bp$ to some vertex $z\in\rm$, then $x,y,z$ would span a $T_3$; thus $P$, and hence also $mCn$, has to visit \rp\ or \bm. Suppose that $mCn$ visits \bm\ before \rp\ and let $u$ be its first vertex in \bm. Note that all vertices of $mCu^-$ lie in $\bp\cup\rm\cup\{x\}$.

Clearly, $x\nodob u^-$ as otherwise $u\dob x\dob u^-$. This means that $u^-\notin\bp\cup \rbm\cup\{x\}$, hence $u^- \in \rrm$. Moreover, $m\nodob u^-$  as $m\dob u^-$ would imply $x\dob m\dob u^-$. As $m\dor u^-$ would imply $m\dor u^-\dor x$, contradicting the fact that $m\in \bbm$, we obtain $m\odoy u^-$. By \Lr{xCy} we have $m\odoy u^-$ also in $mCu^-$. But any \yel\ path in $mCu^-$ from $m$ to $u^-$ has to leave \bp\ for the first time at some point. As there is no \yel\ edge entering $x$, it has to do so along an edge to \rm. However, the endvertices of such an edge together with $x$ would form a $T_3$. This contradicts our assumption that $mCn$ visits \bm\ before \rp.

Thus $mCn$ visits \rp\ at some point without having visited  \bm\ before. Let $r$ be the first vertex of $mCn$ that lies in \rp. We will show that it lies in \rrp. Since $m \odob x$, \Lr{xCy} yields that $m\odob x$ also in $D[mCx]$, hence $mCx$ has to meet \bm. As $mCr$ does not meet \bm, we have $x\notin mCr$, and in particular $r^-\not= x$. This means that $r^-\in\bp\cup\rm$. Suppose that $r^-\in \rrm$. Now $m\dor r^-$ would imply $m\dor r^-\dor x$, and $m\dob r^-$ would imply $x\dob m\dob r^-$, contradicting the fact that $m\in \bbp$ and $r^-\in \rrm$. Hence $m\odoy r^-$ in this case, so \Lr{xCy} implies that $m\odoy r^-$ also in $D[mCr^-]$. But $mCr^-$ only meets vertices in \bp\ and \rm, thus there has to be a \yel\ edge from \bp\ to \rm, which again yields a $T_3$. This contradiction shows that $r^-\notin \rrm$.

As $r\in\rp$, we have $r\not= x$ and thus $r^-\not= x^-$. Hence $r^-\in\bp\cup \rbm$, in particular $x\dob r^-$. Now $r\dob x$ would contradict $r\nodom r^-$, whence $r\in \rrp$, and we can choose $p=r$.
\end{proof}

\begin{corollary} \label{geil}
For every pair of vertices $m \in \bbp$ and $n \in \rrm$ with $m\not= n^+$, the path $mCn$ contains a subpath $pCt$ with $p \in \rrp$ and $t \in \bbm$.
 \end{corollary}

\begin{proof}
We can apply \Prr{r} to obtain the vertex $p$. By changing the direction of every edge (note that by \Lr{reverse} this operation preserves the fact that the tournament is a minimal counterexample), switching the colours blue and red and applying \Prr{r} again with the roles of $m$ and $n$ interchanged, we find the vertex $t \in \bbm$. As $mCp$ does not meet \bm, the path $mCn$ meets $p$ before $t$, hence $pCt$ is contained in $mCn$.
\end{proof}



Applying \Cr{geil} repeatedly we can now prove the main result of this section.

\begin{proof}[Proof of \Tr{3colours}]
  Suppose, to the contrary, there is a vertex $x$ as described at the beginning of this section. Lemmas~\ref{m} and~\ref{n} yield vertices $m\in \bbp$ and $n\in \rrm$ with $m\not= n^+$. Applying \Lr{geil} yields a subpath $pCt$ of $mCn$ with $p\in \rrp$ and $t\in \bbm$. As, clearly, $p\not= t^+$, we can apply \Lr{geil} again, this time to $pCt$ instead of $mCn$ and with the roles of the colours red and blue interchanged, to obtain a subpath $m_1Cn_1$ of $pCt$ with $m_1\in \bbp$ and $n_1\in\rrm$. We can keep on applying \Cr{geil} again and again, to obtain a sequence of nested paths $mCn \geq m_1 C n_1 \geq m_2 C n_2 \ldots$, contradicting the fact that $D$ is finite.
\end{proof}

As discussed earlier, it is easy to see that a minimal counterexample to \Tr{main} is also a minimal counterexample to \Cnr{conj}. Thus \Tr{3colours} immediately implies \Tr{main}.

\acknowledgement{We would like to thank Henning Bruhn for very helpful discussions on this problem.}

\bibliographystyle{plain}
\bibliography{collective}

\end{document}